\documentclass[a4paper,11pt]{article}

\usepackage{amsmath,amssymb,amsthm,mathrsfs}
\usepackage[utf8]{inputenc}
\usepackage{a4}
\usepackage[english]{babel}
\usepackage{enumerate}
\usepackage{graphicx}
\usepackage{color}
\usepackage[
  bookmarks=true,
  bookmarksopen=true,
  bookmarksnumbered=true,
  pdfauthor={Dominic Descombes},
  pdftitle={Asymptotic rank of spaces with bicombings},
  pdfsubject={Paper},
  pdfstartview=FitH,
  colorlinks=false,
  pdfencoding=unicode
]{hyperref}

\numberwithin{equation}{section}

\newtheorem{Thm}{Theorem}[section]
\newtheorem{Prop}[Thm]{Proposition}
\newtheorem{Lem}[Thm]{Lemma}
\theoremstyle{definition}
\newtheorem{Expl}[Thm]{Example}
\newtheoremstyle{customNumber}
     {}          
     {}          
     {\itshape}  
     {}          
     {\bfseries} 
     {.}         
     { }         
     {\thmname{#1}\thmnumber{ #2}\thmnote{ #3}}
\theoremstyle{customNumber}
\newtheorem*{ThmCustomNumber}{Theorem}

\newcommand{\Q}{\mathbb{Q}}
\newcommand{\R}{\mathbb{R}}
\newcommand{\Z}{\mathbb{Z}}
\newcommand{\N}{\mathbb{N}}
\newcommand{\bary}{\operatorname{bar}}
\newcommand{\CAT}{\operatorname{CAT}}
\newcommand{\del}{\delta}

\newcommand{\eps}{\varepsilon}

\newcommand{\gam}{\gamma}
\newcommand{\lam}{\lambda}
\renewcommand{\phi}{\varphi}
\renewcommand{\rho}{\varrho}
\newcommand{\sig}{\sigma}
\newcommand{\olX}{\,\overline{\!X}}
\newcommand{\diam}{\operatorname{diam}}
\newcommand{\di}{\partial}

\title{Asymptotic rank of spaces with bicombings}

\author{Dominic Descombes}

\date{October 19, 2015}

\begin{document}


\pdfbookmark[0]{Asymptotic rank of spaces with bicombings}{titleLabel}
\maketitle

\begin{abstract}
The question, under what geometric assumptions on a space $X$ an $n$-quasiflat in $X$ implies the existence of an $n$-flat therein, has been investigated for a long time.
It was settled in the affirmative for Busemann spaces by Kleiner, and for manifolds of non-positive curvature it dates back to Anderson and Schroeder. We generalize the theorem of Kleiner to spaces with bicombings. This structure is a weak notion of non-positive curvature, not requiring the space to be uniquely geodesic. Beside a metric differentiation argument, we employ an elegant barycenter construction due to Es-Sahib and Heinich by means of which we define a Riemannian integral serving us in a sort of convolution operation.
\end{abstract}
 

\section{Introduction}

By a 
{\em geodesic bicombing\/}\label{def:bicombing} $\sig$ on a metric space $(X,d)$ we mean a map 
\[
\sig \colon X \times X \times [0,1] \to  X
\]
that singles out a constant speed geodesic
$\sig_{xy} := \sig(x,y,\cdot)$ from $x$ to $y$ (that is, 
$d(\sig_{xy}(t),\sig_{xy}(t')) = |t - t'|d(x,y)$ for all $t,t' \in [0,1]$, 
and $\sig_{xy}(0) = x$, $\sig_{xy}(1) = y$) for every pair $(x,y) \in X \times X$. 
The weak non-positive curvature condition we impose on $\sig$ is the {\em conical\/} property
\begin{equation} \label{eq:conical}
d(\sig_{xy}(t),\sig_{x'y'}(t)) \le (1-t)\,d(x,x') + t\,d(y,y')
\;\, \text{for all $\,t \in [0,1]$} 
\end{equation}
and $x,y,x',y' \in X$. 
Observe that this does not imply that $t\mapsto d(\sig_{xy}(t),\sig_{x'y'}(t))$ is a convex function on $[0,1]$ (which would be the case in Busemann and $\CAT(0)$ spaces).
Because we work exclusively with conical geodesic bicombings, we suppress the first part throughout; a {\em bicombing} is thus implicitly understood to be conical and geodesic.
Also note that the conical property implies that $\sig$ is continuous and, in particular, $X$ 
is contractible via \mbox{$X\times[0,1]\to X$}, $(x,t)\mapsto \sig_{ox}(1-t)$ for an arbitrary choice of basepoint $o\in X$. Furthermore, we will say that a bicombing $\sig$ is {\em reversible\/} if 
\begin{equation} \label{eq:reversible}
\sig_{xy}(t) = \sig_{yx}(1-t) \quad \text{for all $t \in [0,1]$} 
\end{equation}
and $x,y \in X$. Occasionally, when readability demands it, we will write $[x,y]$ instead of $\sig_{xy}$ for the geodesics of a bicombing. And, when using this notation, we do not explicitly distinguish between $[x,y]$ as a map and its image (or trace) as long as no confusion arises.
Next, let us mention here that there is no loss of generality in assuming a space $X$ with (reversible) bicombing $\sig$ to be complete. As for two points $x,y$ in the completion $\olX$, we may (and must) set $\phantom{\,}\overline{\!\sig}_{xy}:=\lim_{k\to\infty} \sig_{x_ky_k}$ given two sequences $x_k\to x, y_k\to y$ with $x_k,y_k\in X$. For every $t\in[0,1]$, the sequence $\sig_{x_ky_k}(t)$ is then Cauchy and the convergence $\sig_{x_ky_k}\!\to\overline{\!\sig}_{xy}$ is uniform and the limit independent of the chosen sequences by \eqref{eq:conical}. Finally, and since the main theorem is about proper spaces, let us note here (and prove at the end of this introduction) that given a proper space with bicombing, we may always construct a new bicombing that is reversible. For a systematic study of spaces with bicombings we refer to \cite{DesL1,DesL2}.

The wording asymptotic rank for a space $X$ was coined by Wenger in \cite{Wen} and is (equivalently) defined to be the supremum over all $n$ for which there are sequences $R_k\in(0,\infty)$, $S_k\subset X$, and a normed vector space $(\R^n,\|\cdot\|)$, so that $R_k\to\infty$, and $\frac{1}{R_k}S_k$ converges to the unit ball $B\subset(\R^n,\|\cdot\|)$ in the Gromov\/-Hausdorff topology --- compare the assumptions of Theorem~\ref{Thm:D} below. The unit ball $B$ is then said to be an asymptotic subset of $X$. Preceding the work of Wenger and in the context of Busemann spaces, Kleiner showed an equivalence between this asymptotic rank and the more classical Minkowski rank defined as follows. The {\em Minkowski rank} of a space $X$ is the supremum over all $n$ for which there is a $n$-dimensional normed space embedding isometrically into $X$.
The main result for Busemann spaces, the one we are about to generalize, is Theorem D in \cite{Kle}. Theorem~\ref{Thm:D} forms the missing piece to carry over (most of) Kleiner's Theorem~D to our setting and generalizes the analogous Proposition~10.22 in \cite{Kle} (with a slight improvement even in the case of Busemann spaces).
\begin{Thm}\label{Thm:D} Let $X$ be a proper metric space with a bicombing $\sig$ and
cocompact isometry group.
Suppose there are sequences $R_k\in(0,\infty)$, $S_k\subset X$, and a normed vector space $(\R^n,\|\cdot\|)$, so that $R_k\to\infty$, and $\frac{1}{R_k}S_k$ converges to the unit ball $B\subset(\R^n,\|\cdot\|)$ in the Gromov\/-Hausdorff topology. Then $(\R^n,\|\cdot\|)$ can be isometrically embedded in $X$. 
\end{Thm}
The original proof relies in a crucial way upon properties of Busemann functions that are no longer available in our context. (Even for such a simple case as $\R^2$ with the supremum norm, the indefinitely many (linear) rays issuing from $0\in\R^2$ produce only $8$ Busemann functions.) Therefore we take a quite different road and make use of the bicombing geodesics only to provide us with a barycenter map by means of which we define a Riemannian integral for maps with values in spaces with bicombings.
After the proof of Theorem~\ref{Thm:D}, which involves a metric differentiation argument and a convolution like operation, we finally state the analogue to Kleiner's Theorem --- Theorem~\ref{Thm:Dreally}. Among other things, this theorem shows the following. Every proper cocompact metric space $X$ with a 
bicombing contains an $n$-flat (an isometrically embedded $n$-dimensional normed space) whenever it contains an $n$-quasiflat (a quasi-isometrically embedded $n$-dimensional normed space).
The first result of that kind was shown for manifolds of non-positive curvature in \cite{AndS}.

In Section~\ref{Chapter:Barycenters} we preliminarily develop the barycenter map and then devote Section~\ref{Chapter:Asymptotic rank} to the proofs of \ref{Thm:D} and \ref{Thm:Dreally}.
But first we end this section with the prior mentioned statement about passing from a possibly non-reversible bicombing to a reversible one.

\begin{Prop}\label{Prop:forceReversible} If a proper metric space $X$ admits a (possibly non-reversible) bicombing $\sig$, then $X$ admits a reversible bicombing $\tilde\sig$.
\end{Prop}

\begin{proof} First we claim that the existence of a bicombing is equivalent to the existence of a midpoint assignment $(x,y)\mapsto x\#y$ (i.e.\ $x\#y$ is a point such that $d(x,x\#y)=d(x\#y,y)=\tfrac{1}{2}d(x,y)$) with the additional (conical) property
\begin{equation}\label{eq:midpointconical}
d(x\#y,x'\#y')\leq \frac12 d(x,x') + \frac12 d(y,y') ;
\end{equation}
and this equivalence holds for any complete space $X$ (which every proper space is). Obviously $(x,y)\mapsto\sig_{xy}(1/2)$ is a valid midpoint assignment with \eqref{eq:midpointconical} given a bicombing $\sig$. To define a bicombing from a midpoint assignment, set $\sig_{xy}(0):=x,\sig_{xy}(1):=y$, and $Q_0:=\{0,1\}$. Then for every $k\geq 1$ let $Q_k:=\{ m/2^k \,|\, m\text{ odd }, 0<m<2^k \}$, and inductively define $\sig_{xy}(t)$ at $t\in Q_k$ to be $\sig_{xy}(t-2^{-k})\#\sig_{xy}(t+2^{-k})$. See that $t-2^{-k},t+2^{-k}$ lie in $\cup_{0\leq i\leq k-1}Q_i$, and by induction it is easy to verify \eqref{eq:conical} for $t\in Q:=\cup_{k\in\N_0}Q_k$ as  well as $d(\sig_{xy}(t),\sig_{xy}(t'))=|t-t'|d(x,y)$ for all $t,t'\in Q$.
So the maps $\sig_{xy}\colon Q\to X$ defined so far are Lipschitz continuous for every pair $(x,y)$, hence extend uniquely to geodesics $\sig_{xy}\colon[0,1]\to X$ (preserving \eqref{eq:conical}) as $X$ is complete.

In view of this it suffices to show that, starting from $\#$, we can construct a symmetric midpoint assignment $\Box$, i.e.\ one with $x\,\Box\,y=y\,\Box\,x$ for all $x,y\in X$.
For a pair $(x,y)\in X^2$ we define the sequences $x_i,y_i$ by
\[ x_0=x, \; y_0 = y \;\;\text{and}\;\; x_i=x_{i-1}\#\;\!y_{i-1}, \; y_i=y_{i-1}\#\;\!x_{i-1} \;\; \text{for all } i\geq 1 .
\]
We have $d(x_{i+1},y_{i+1})\leq d(x_i,y_i)$ by applying \eqref{eq:midpointconical} to the definition of the sequences as well as
\[
d(x_i,x_j) = d(x_j,y_i)=\frac{d(x_i,y_i)}{2}=d(x_i,y_j)=d(y_j,y_i)\,\text{ for all $j > i$}.
\]
The latter results by induction over $j$ and the fact that $a,b\in M(x,y)$ implies $a\#b$, $b\#a\in M(x,y)$ where $M(x,y):=\{z\,|\,d(x,z)=d(z,y)=\tfrac12 d(x,y)\}$ denotes the  set of all midpoints between $x$ and $y$.
Consequently, if the monotone sequence $d(x_i,y_i)$ would not converge to zero, then $d(x_i,x_j)\geq \eps$ for all $i,j$ and some $\eps>0$ thereby contradicting compactness. So both sequences converge to a common limit and we set $x\,\Box\,y=\lim_{i\to\infty}x_i=\lim_{i\to\infty}y_i = y\,\Box\,x$. The verification that $\Box$ obeys \eqref{eq:midpointconical} follows when taking the limit $i\to\infty$ of either of the inequalities
\begin{align*}
d(x_i,x'_i) & \leq \frac12 d(x,x') + \frac12 d(y,y'), \\
d(y_i,y'_i) & \leq \frac12 d(x,x') + \frac12 d(y,y'),
\end{align*}
where $x'_i,y'_i$ are the sequences in the construction of $x'\Box y'$. And these inequalities are easily shown by mutual induction.
\end{proof}


\section{Barycenters}\label{Chapter:Barycenters}

Now we develop the tool that forms a crucial component in the proof of the main result. 
In~\cite{EsSH}, Es-Sahib and Heinich introduced an elegant 
barycenter construction for Busemann spaces, which was reviewed and partly 
improved in a recent paper by Navas~\cite{Nav}. The construction and proofs 
translate almost verbatim to spaces with reversible bicombings. As such spaces may lack unique midpoints, the only modification required is to set $\bary_2(x,y):=\sig_{xy}(1/2)$. After reviewing the construction we define the integration mentioned previously.

We need two more definitions before we can start. For an isometry $\gam\colon X\to X$ and a bicombing $\sig$ on $X$ we say that $\sig$ is {\em $\gam$-equivariant} if $\gam\circ[x,y]=[\gam(x),\gam(y)]$ for all $x,y\in X$. A set $C\subset X$ is called $\sig$-convex if for all $x,y\in C$ also the trace $[x,y]$ is contained in $C$. The {\em (closed) $\sig$-convex hull} of $C$ is the intersection of all closed and $\sig$-convex sets containing $C$, thus the smallest such set. Equivalently, the $\sig$-convex hull may  be written as the closure of
$\cup_{k\in\N_0} C_k$ where $C_0=C$ and $C_k = \cup_{x,y\in C_{k-1}}[x,y]$ for $k\geq 1$. From \eqref{eq:conical} we have $\diam(C_{k-1})=\diam(C_k)$. Moreover, taking the closure does not increase diameter but preserves $\sig$-convexity, so the diameters of $C$ and the $\sig$-convex hull of $C$ coincide.

\begin{Thm}\label{Thm:Bary} For every complete space $X$ with reversible bicombing $\sig$ and every $n\in\N$ there is a barycenter map $\bary_n\colon X^n\to X$ with the following properties:
\begin{enumerate}[(i)]
\item\label{Thm:Bary:inHull} $\bary_n(x_1,\ldots,x_n)$ lies in the $\sig$-convex hull of $\{x_1,\ldots,x_n\}$,
\item\label{Thm:Bary:Perm} $\bary_n$ is permutation invariant, i.e.\ for any permutation $\pi\in S_n$ we have $\bary_n(x_1,\ldots,x_n)=\bary_n(x_{\pi(1)},\ldots,x_{\pi(n)})$,
\item\label{Thm:Bary:Equi} $\gam(\bary_n(x_1,\ldots,x_n)) = \bary_n(\gam (x_1),\ldots,\gam (x_n))$ for every isometry $\gam$ of $X$ provided $\sig$ is $\gam$-equivariant,
\item \label{Thm:Bary:Cont}$d\big(\!\bary_n(x_1,\ldots,x_n),\bary_n(y_1,\ldots,y_n)\big) \leq \underset{\pi\in S_n}{\min} \frac1n\sum_{i=1}^n d(x_i,y_{\pi(i)})$.
\end{enumerate}
\end{Thm}

Note that for $n=2$, $\bary_2$ as above is exactly what we call a symmetric midpoint assignment. And (in view of Proposition~\ref{Prop:forceReversible}) if the given space is proper, even a (possibly non-reversible) bicombing (or midpoint assignment) is enough to obtain these barycenters. Also observe that \eqref{Thm:Bary:inHull} implies $\bary_n(x,\ldots,x)=x$ for every $n$ and $x$ as we would expect.

\begin{proof} First, property \eqref{Thm:Bary:inHull} forces us to set $b_1(x_1)=x_1$. Second, we define $\bary_2(x_1,x_2)=\sig_{x_1x_2}(1/2)$. This definitions are in accord with all the required properties since the bicombing is assumed the be reversible. For $n\geq 3$, let $x_1^0=x_1,x_2^0=x_2,\ldots,x_n^0=x_n$ be the initial points and recursively
\[
x_j^i := \bary_{n-1}(x_1^{i-1},\ldots,\widehat{x_j^{i-1}},\ldots,x_n^{i-1}) \;\; \text{for $i\geq 1$} ,
\]
where $\widehat{x_j^{i-1}}$ means that this point is omitted. We proceed by induction over $n$ --- assuming the properties for $\bary_{n-1}$ --- and claim that $\diam(\{x_1^i,\ldots,x_n^i\})$ converges to zero for $i\to\infty$. This diameter equals the one for the $\sig$-convex hull of $\{x_1^i,\ldots,x_n^i\}$, and this set is clearly contained in the hull of $\{x_1^{i-1},\ldots,x_n^{i-1}\}$. The resulting monotonic sequence converges to a single point which we declare to be the value of $\bary_n(x_1,\ldots,x_n)$; equivalently let it be the limit of $x_j^i$ when $i\to\infty$ (the choice of $j$ is obviously irrelevant).

For the claim, observe that $d(x_k^i,x_l^i)\leq\tfrac{1}{n-1} d(x_k^{i-1},x_l^{i-1})$ from \eqref{Thm:Bary:Cont}. Hence $\diam(\{x_1^i,\ldots,x_n^i\})\leq\tfrac{1}{n-1}\diam(\{x_1^{i-1},\ldots,x_n^{i-1}\})$ and $\bary_n$ is well-defined. Since all points $x_j^i$ lie in the $\sig$-convex hull of $\{x_1,\ldots,x_n\}$, so does the $\bary_n(x_1,\ldots,x_n)$. \eqref{Thm:Bary:Perm} results from the permutation invariant nature of the construction, and since --- starting from the $x_j^i$ for $\bary_n(x_1,\ldots,x_n)$ --- $\,\gam(x_j^i)$ are the points that arise in the construction of $\bary_n(\gam(x_1),\ldots,\gam(x_n))$, we obtain \eqref{Thm:Bary:Equi} as well. Finally, \eqref{Thm:Bary:Equi} reduces \eqref{Thm:Bary:Cont} to $d\big(\!\bary_n(x_1,\ldots,x_n),\bary_n(y_1,\ldots,y_n)\big) \leq  \frac1n\sum_{i=1}^n d(x_i,y_i)$. By the inductive assumption we have
\[
d(x_j^i,y_j^i)\leq\frac{1}{n-1}\mathop{\sum_{k=1}}_{k\neq j}^n d(x_k^{i-1},y_k^{i-1})
\]
and taking the sum over $j$ thereof, we arrive at
\[
\frac{1}{n}\sum_{j=1}^n d(x_j^i,y_j^i)\leq\frac{1}{n}\sum_{j=1}^n d(x_j^{i-1},y_j^{i-1})\leq\cdots\leq \frac{1}{n}\sum_{j=1}^n d(x_j^0,y_j^0) .
\]
The leftmost term converges to the distance we are about to estimate (when $i\to\infty$) and thereby completes the proof. 
\end{proof}

It is clear from the construction of the points $x_j^i$ in the proof above that $\bary_n(x_1,\ldots,x_n)=\bary_n(x_1^1,\ldots,x_n^1)$, where $x_j^1=\bary_{n-1}(x_1,\ldots,\widehat{x_j},\ldots,x_n)$. This relation leads to the subsequent estimate we will use in Proposition~\ref{Prop:BaryCauchy}, which is the key step on route to Theorem~\ref{Thm:BaryMeasure}.

\begin{Lem}\label{Lem:Bary:Est} For the barycenters of Theorem~\ref{Thm:Bary} above and any $x\in X$ and integers $1\leq k\leq n$ we have
\[
d(x, \bary_n(x_1,\ldots,x_n))\leq \binom{n}{k}^{\!-1}\!\!\!\! \mathop{\sum_{I\subset\{1,\ldots,n\}}}_{|I|=k} d(x,\bary_k(x|_I)) .
\]
Here $\bary_k(x|_I)$ denotes the barycenter of the $k$-tuple whose entries are the $k$ (not necessarily distinct) values $x_i$ for $i\in I$ (and the tuples order is irrelevant as shown before).
\end{Lem}

\begin{proof} Fixing $k$, we apply induction over $n$. For $k=n$ the statement is trivial and our base case; hence assume $n>k$ and the statement to  hold for \mbox{$n-1\geq 1$}. By means of Theorem~\ref{Thm:Bary}\eqref{Thm:Bary:Cont} together with $x=\bary_n(x,\ldots,x)$ and the previously mentioned relation, we obtain
\begin{align*}
d(x, \bary_n(x_1,\ldots,x_n)) & \leq \frac{1}{n}\sum_{j=1}^n d(x,x_j^1) \\
& =\frac{1}{n}\sum_{j=1}^n d\big(x,\bary_{n-1}(x_1,\ldots,\widehat{x_j},\ldots,x_n)\big) .
\end{align*}
We may estimate the last expression by
\[
\frac{1}{n} \binom{n-1}{k}^{\!-1}\sum_{j=1}^n\; \mathop{\sum_{I\subset\{1,\ldots,n\}\setminus\{j\}}}_{|I|=k} d(x,\bary_k(x|_I))
\]
using the induction hypothesis.
For a fixed $I\subset\{1,\ldots,n\}, |I|=k$ the double sum counts $d(x,\bary_k(x|_I))$ exactly $n-k$ times. Since we have the identity
$
\frac{1}{n} {n-1\choose k}^{-1}(n-k) = {n\choose k}^{-1}
$
for binomial coefficients, this concludes the proof.
\end{proof}

One very natural property may not hold for the barycenters constructed so far. If for an $n$-tuple $\mathbf{x}=(x_1,\ldots,x_n)$ we write $k\cdot \mathbf{x}$ for the $kn$-tuple where every entry is repeated $k$ times, e.g. $3\cdot \mathbf{x} = (x_1,x_1,x_1,x_2,x_2,x_2,\ldots,x_n,x_n,x_n)$, then it may not be the case that
\begin{equation}\label{Eq:BaryDefect}
\bary_n(\mathbf{x})=\bary_{kn}(k\cdot \mathbf{x}) .
\end{equation}

\begin{Expl}\label{Expl:BaryDefect} Let $K$ be the tree that emerges from gluing two copies of $[0,1]$ to $[0,2]$ identifying all three zeros to become the single point $m$.
The unique geodesics in $K$ form a bicombing (in fact this is a $\CAT(0)$ space).
Now if $y,z$ are the endpoints of the  branches of length $1$ and $x$ the endpoint of the remaining branch of length $2$, then $\bary_3(x,y,z)$ is at distance $1/3$ from $m$ on the geodesic to $x$, but $\bary_6(x,x,y,y,z,z)$ lies at distance $13/45$ from $m$ on the same geodesic.
\end{Expl}

See that one can facilitate many computations of barycenters in a space $X$ with bicombing by detecting the following pattern. Whenever there is a subset $Y\subset X$ being $\sig$-convex, isometric to a convex subset of a normed vector space via $\phi\colon Y\to V$, and such that $\phi^{-1}((1-\lam)\phi(x)+\lam\phi_y)$ equals the restricted bicombing on $Y$; then the above construction yields a barycenter (on $Y$) behaving like the linear barycenter $\bary_n(x_1,\ldots,x_n):=\tfrac{1}{n}\sum_{i=1}^n x_i$ on $\phi(Y)$. Thus confronted with the task of computing $\bary_n(x_1,\ldots,x_n)$ in $X$ when $x_1,\ldots,x_n\in Y$, we may as well compute the linear barycenter of $\phi(x_1),\ldots,\phi(x_n)$ and take the preimage. This remark is of course rather trivial since a map $\phi$ as above transports one bicombing into the other, and therefore $Y$ and $\phi(Y)$ are indistinguishable from our point of view.

The fact that \eqref{Eq:BaryDefect} may not hold prevents us to define the barycenter for probability measures with, say, finite support. To address this defect, one is tempted to define
\[
\bary^*_n(\mathbf{x}):=\lim_{k\to\infty}\bary_{kn}(k\cdot \mathbf{x})
\]
and hope this limit always exists. One of the central observations of \cite{EsSH,Nav} is that this in fact works, and we present a streamlined version of a proof found in \cite{Nav}, which we were able to simplify considerably by use of elementary statistics.

\begin{Prop}\label{Prop:BaryCauchy} Still in the setting of Theorem~\ref{Thm:Bary}, let the barycenters there be given. Then for any $n$ and $n$-tuple $\mathbf{x}$ we have that $k\mapsto\bary_{nk}(k\cdot \mathbf{x})$ is a Cauchy sequence in $X$ and thus convergent.
\end{Prop}

\begin{proof} Let $n,k$ be positive integers and $\mathbf{x}\in X^n$ an $n$-tuple. We want to estimate $d\big(\bary_{kn}(k\cdot \mathbf{x}),\bary_{(k+l)n}((k+l)\cdot \mathbf{x})\big)$ for arbitrary $l\geq 1$; by Lemma~\ref{Lem:Bary:Est} this is less or equal to
\[
\binom{(k+l)n}{kn}^{-1}\!\!\!\! \mathop{\sum_{I\subset\{1,\ldots,(k+l)n\}}}_{|I|=kn} d\!\left(\bary_{kn}(k\cdot \mathbf{x}),\,\bary_{kn}\!\left(\big[(k+l)\cdot \mathbf{x}\big]_I\right)\!\right) .
\]
For an $I$, over which the above sum runs, let $i_1,\ldots,i_n$ be defined as follows: $i_j$ counts how many times $x_j$ appears in $[(k+l)\cdot \mathbf{x}]_I$. So $[(k+l)\cdot \mathbf{x}]_I$ starts with $i_1$ copies of $x_1$ followed by $i_2$ copies of $x_2$ and so on. Now the continuity property \ref{Thm:Bary}\eqref{Thm:Bary:Cont} of the barycenters leads to
\begin{align*}
d\!&\left(\bary_{kn}(k\cdot \mathbf{x}),\,\bary_{kn}\!\left(\big[(k+l)\cdot \mathbf{x}\big]_I\right)\!\right) \\
& \quad\quad \leq \frac{D}{2kn}\left(|i_1-k|+|i_2-k|+\cdots+|i_n-k| \right) .
\end{align*}
For every specific assignment of the $i_j$ there are exactly
$
{k+l\choose i_1}{k+l\choose i_2}\cdots{k+l\choose i_n}
$
sets $I$ that produce these values. Consequently our upper bound now reads
\[
\mathop{\sum_{0\leq i_1,\ldots,i_n\leq k+l}}_{i_1+\cdots+i_n=kn} \frac{{k+l\choose i_1}\cdots{k+l\choose i_n}}{\binom{(k+l)n}{kn}} \frac{D}{2kn}\left(|i_1-k|+|i_2-k|+\cdots+|i_n-k| \right) .\,\footnote{This is (essentially) the first absolute centric moment of a multivariate hypergeometric distribution, and a simplified expression for this exists which would end the proof here. Since that expression is rather hard to find in the literature, we choose not to use it.}
\]
By symmetry of this expression we may replace $\sum_{j=1}^n |i_j-k|$ by $n|i_1-k|$. Running the sum over $i_2,\ldots,i_n$ and exploiting Vandermonde's identity we simplify further to
\[
\frac{D}{2k}\sum_{i=0}^{k+l} \frac{{k+l\choose i}{(k+l)(n-1)\choose kn-i}}{\binom{(k+l)n}{kn}} |i-k| .
\]
The fraction of binomial coefficients is the probability mass function of a hypergeometric distribution with parameters $kn$ (draws without replacement), $(k+l)n$ (balls in the urn) and $k+l$ (balls that count). Let $Y$ be a random variable distributed accordingly; the mean value is ${\rm E}[Y]=k$ and the variance may be estimated  by ${\rm Var}[Y]\leq k$. The above expression is therefore equal to $\tfrac{D}{2k}{\rm E}[|Y-{\rm E}[Y]|]$ and by Jensen's inequality
\[
\frac{D}{2k}{\rm E}[|Y-{\rm E}[Y]|]\leq \frac{D}{2k}\sqrt{{\rm E}[(Y-{\rm E}[Y])^2]} \leq \frac{D}{2\sqrt{k}}.
\]
So we have a bound for $d\big(\bary_{kn}(k\cdot \mathbf{x}),\bary_{(k+l)n}((k+l)\cdot \mathbf{x})\big)$ independent of $l$ and going to zero for $k\to\infty$ as sought.
\end{proof}

Now we have everything in place to define barycenters for probability measures with finite support (at this point the barycenters could be easily defined for a much wider class of measures, see again \cite{EsSH,Nav}; we skip this since we have no use for it in the sequel). Every such measure $\mu$ on a metric space $X$ can be written uniquely as a sum $\sum_{i=1}^k a_i\del_{x_i}$ of Dirac measures (point measures) where $x_1,\ldots,x_k$ are pairwise distinct points and $a_1,\ldots,a_k> 0$, $a_1+\cdots+a_k=1$. Given another measure $\nu=\sum_{j=1}^l b_j\del_{y_j}$, a {\em mass transport from $\mu$ to $\nu$} shall be a function $T\colon\{1,\ldots,k\}\times\{1,\ldots,l\}\to[0,\infty)$ provided $\sum_{i=1}^kT(i,j)=b_j$ for all $j$ and $\sum_{j=1}^lT(i,j)=a_i$ for all $i$. By the {\em cost} $\,c(T)$ of a mass transport we understand the real value $\sum_{i,j} T(i,j)d(x_i,y_j)$.
We define the Wasserstein distance $d_W(\mu,\nu)$ to be the infimum over the costs of all mass transports from $\mu$ to $\nu$, and it is not hard to verify that this is indeed a metric and that the infimum is always attained for some transport $T$ which we call an {\em optimal transport}.
(There are of course much more sophisticated definitions of this distance but this one is enough for our purpose and relatively easy to handle.) Note that the bound in estimate \ref{Thm:Bary}\eqref{Thm:Bary:Cont} is exactly the Wasserstein distance for $\mu=\sum_{i=1}^n \tfrac{1}{n}\del_{x_i}$ and $\nu=\sum_{j=1}^n \tfrac{1}{n}\del_{y_j}$. To see this, fix an optimal transport $T$ and consider the bipartite graph on the disjoint union $\{1,\ldots,n\}\dot\cup\{1,\ldots,n\}$ where $i$ in the first set is connected to $j$ in the second if $T(i,j)>0$. This graph fulfills the premises of Hall's marriage theorem, and hence there is a permutation $\pi$ such that $T(i,\pi(i))>0$ for every $i$. If we define the transport $\tilde T$ to be equal to $T$ except for the points $(i,\pi(i))$ where we subtract $D:=\min_{i=1,\ldots,n}T(i,\pi(i))$, we obtain
\[
c(T)=c(\tilde T)+\sum_{i=1}^n Dd(x_i,y_{\pi(i)}) .
\]
$\tilde T$ is an (optimal) transport from $(1-nD)\mu$ to $(1-nD)\nu$ so $c(\tilde T)=(1-nD)c(T)$ and therefore the distances coincide.

\begin{Thm}\label{Thm:BaryMeasure} For every complete space $X$ with reversible bicombing there is a barycenter map $\bary$ assigning each probability measure $\mu$ with finite support $|\operatorname{spt}(\mu)|<\infty$ a barycenter $\bary(\mu)\in X$. This map has the properties
\begin{enumerate}[(i)]
\item\label{Thm:BaryMeasure:inHull} $\bary(\mu)$ lies in the $\sig$-convex hull of $\operatorname{spt}(\mu)$,
\item\label{Thm:BaryMeasure:Equi} $\gam(\bary(\mu)) = \bary(\gam_*\mu)$ for every isometry $\gam$ of $X$ provided $\sig$ is $\gam$-equivariant (where $\gam_*\mu:=\sum_{i=1}^l a_i\del_{\gam(x_i)}$ denotes the push-forward of $\mu=\sum_{i=1}^l a_i\del_{x_i}$),
\item\label{Thm:BaryMeasure:Cont} $d(\bary(\mu),\bary(\nu))\leq d_W(\mu,\nu)$.
\newcounter{Thm:BaryMeasure:counter}
\setcounter{Thm:BaryMeasure:counter}{\value{enumi}}
\end{enumerate}
\end{Thm}

\begin{proof} First let $\bary_n$ be the maps provided by Theorem~\ref{Thm:Bary} and define $\bary_n^*$ by virtue of the previous proposition:
\[
\bary_n^*(x_1,\ldots,x_n):=\lim_{k\to\infty} \bary_{kn}(k\cdot \mathbf{x}),\;\;\text{where }\, \mathbf{x}=(x_1,\ldots,x_n).
\]
It is straightforward to verify properties \eqref{Thm:Bary:inHull} through \eqref{Thm:Bary:Cont} of \ref{Thm:Bary} for all $\bary_n^*$. Moreover, now clearly $\bary_n^*(x_1,\ldots,x_n)=\bary_{kn}^*(k\cdot \mathbf{x})$ for every $k\geq 1$. Because of this equality, we may now unambiguously define $\bary(\mu)$ provided the coefficients $a_i$ of $\mu=\sum_{i=1}^l a_i\del_{x_i}$ all lie in $\Q$. Let $m$ be a positive integer such that $ma_i\in\N$ for all $i$, then $\bary(\mu):=\bary_m(x)$ where $x$ is a tuple starting with $ma_1$ copies of $x_1$, then $ma_2$ copies of $x_2$, et cetera. \eqref{Thm:BaryMeasure:inHull} and \eqref{Thm:BaryMeasure:Equi} are clear from this definition and \eqref{Thm:BaryMeasure:Cont} follows from the remark preceding this proof. Finally, measures with rational coefficients are dense among all probability measures with finite support\footnote{$d_W(\sum_{i=1}^la_i\delta_{x_i},\sum_{i=1}^lb_i\delta_{x_i})\leq\tfrac{1}{2}\diam(\{x_1,\ldots,x_l\})\sum_{i=1}^l|a_i-b_i|$} and the properties for the uniquely extended map are once again not hard to verify.
\end{proof}

We now turn to the definition of the Riemannian integral for maps with values in spaces with reversible bicombings. This is a preparation for the second step of the proof of Theorem~\ref{Thm:D}. Let $f\colon M\to X$ be a map from a compact metric space with Borel probability measure $\mu$ to a complete space with reversible bicombing. We call $f$ Riemann integrable if it is bounded and continuous outside a set $D$ of points of discontinuity for which we have $\mu(\overline{\!D})=0$ (equivalently $\mu(U_\delta(D))\to 0$, where $U_\delta$ stands for the open $\delta$-neighborhood). For a finite Borel partition $A_i,i=1,\ldots,k$ of $M$ (that is $\cup_{i=1}^k A_i=M$ and $A_i\cap A_j = \emptyset$ whenever $i\neq j$) and a selection of tagged points $a_i\in A_i$, we define the Riemann sum to be
\[
\bary\!\left(\sum_{i=1}^k\mu(A_i)\delta_{f(a_i)}\right) .
\]
The mesh of a partition is defined to be $\max_{i=1,\ldots,k} \diam(A_i)$. We claim that for every $\eps$ there is a $\delta$ such that if the mesh of two partitions $A_i,B_j$ does not exceed $\delta$, then their Riemann sums are not more than $\eps$ apart. If $K$ is the diameter of the image of $f$, take $\delta$ small enough to ensure $\mu(U_\delta(D))K\leq\eps/4$ and (from uniform continuity on compact subsets) $d(f(x),f(y))\leq\eps/4$ for all $x,y\in M\setminus U_\delta(D)$ with $d(x,y)\leq\delta$. Let $C_{ij}:=A_i\cap B_j$ (dropping empty sets) be the common refinement of the two partitions and subdivide it further into $C_{ij}\cap U_\delta(D)$ and $C_{ij}\setminus U_\delta(D)$. Now a short computation shows that the Riemann sum of the last partition has distance less or equal $\eps/2$ to both, the Riemann sum of $A_i$ and $B_j$. This gives the desired bound. Hence we may define the integral (or barycenter) of $f$ as the limit of Riemann sums when the mesh of the partitions becomes zero (and such exist as $M$ is compact). Furthermore, it is now straightforward to verify the formula
\begin{equation}\label{Eq:IntEst}
d\!\left( \int_M f \,d\mu , \int_M g \,d\mu \right) \leq \int_M d(f(t),g(t)) \,d\mu(t) .
\end{equation}


\section{Asymptotic rank}\label{Chapter:Asymptotic rank}

Now we turn to the main goals, the proofs of Theorem~\ref{Thm:D} and \ref{Thm:Dreally} mentioned in the introduction. We restate the former theorem for convenience and refer to \cite{BriH} for the definition of Gromov-Hausdorff topology (alternatively one may take \eqref{Def:RHA} below for the definition). The proof of \ref{Thm:D} will be realized in three steps; Lemma~\ref{LemStep1}--\ref{LemStep3}.

\begin{ThmCustomNumber}[1.1] Let $X$ be a proper metric space with a bicombing $\sig$ and
cocompact isometry group.
Suppose there are sequences $R_k\in(0,\infty)$, $S_k\subset X$, and a normed vector space $(\R^n,\|\cdot\|)$, so that $R_k\to\infty$, and $\frac{1}{R_k}S_k$ converges to the unit ball $B\subset(\R^n,\|\cdot\|)$ in the Gromov\/-Hausdorff topology. Then $(\R^n,\|\cdot\|)$ can be isometrically embedded in $X$. 
\end{ThmCustomNumber}

Observe that we may assume $\sig$ to be reversible by Proposition~\ref{Prop:forceReversible} and can therefore utilize the barycenter maps. 
For this section, $\|\cdot\|$ always refers to the norm given above and $B(r)$ denotes the norm-ball of radius $r$ center at $0\in\R^n$. The assumptions of the theorem may be restated as follows. There are maps $\Phi_k\colon B(R_k)\to X$ such that
\begin{equation}\label{Def:RHA}
\big| d(\Phi_k(x),\Phi_k(y))-\|x-y\|\big| \leq R_k\eps_k
\end{equation}
for a sequence $R_k$ as above and a null sequence $\eps_k\to 0$. For brevity, we call such a sequence $\Phi_k$ an {\em RHA}; a relative Hausdorff approximation. A priori the maps of an RHA need not be continuous which leads to the first step of the proof of Theorem~\ref{Thm:D}.

\begin{Lem}[Step 1]\label{LemStep1} Given an RHA for a space $X$ with reversible bicombing (and thus a barycenter by means of Theorem~\ref{Thm:Bary}), one may construct an RHA whose maps are all $L$-Lipschitz continuous. More generally, one may require $X$ to be merely Lipschitz $n-1$-connected and the lemma still holds. 
\end{Lem}

\begin{proof}
Set $\lam_k=R_k\eps_k$ and restrict every $\Phi_k$ to those points of $\lam_k\Z^n$ that lie in $B(R_k)$. On this grid (which is a $\delta\lam_k$-separated set for ${\displaystyle \delta:=\min_{x\neq y\in\Z^n}\|x-y\|>0}$) we have
\begin{equation}\label{Eq:GridEst}
d(\Phi_k(x),\Phi_k(y))\leq \|x-y\|+\lam_k\leq \left(1+\frac{1}{\delta}\right)\|x-y\| .
\end{equation}
For every $x\in\lam_k\Z^n$ the set $x+\lam_k[0,1]^n$ is a cube with vertices in the grid. Let $C_k$ be the union of those cubes contained in $B(R_k)$. We extend every restricted map $\tilde\Phi_k$ to $C_k$ through the following procedure: To $y\in x+\lam_k[0,1]^n$ we assign a probability measure supported on $\tilde\Phi_k(x+\lam_k\{0,1\}^n)$ having, for every $e\in\{0,1\}^n$, weight\footnote{The chosen weights are one example of barycentric coordinates in an affine space. Taking the linear combination of these weights with the points $x+\lam_ke$ would recover $y$.}
\[
\prod_{i=1}^n (1-e_i)-(1-2e_i)\!\left(\frac{y_i-x_i}{\lam_k}\right)
\]
at $\tilde\Phi_k(x+\lam_ke)$. The barycenter thereof shall then be the value for $\tilde\Phi_k(y)$. Observe that where the cubes intersect the values coincide. These extensions are $L$-Lipschitz for some constant independent of $k$.
With $\Delta:=\diam_{\|\cdot\|}([0,1]^n)$ we have $B(R_k-\Delta\lam_k)\subset C_k$ and on these sets we get
\[
\big| d(\tilde\Phi_k(x),\tilde\Phi_k(y))-\|x-y\|\big| \leq 2\Delta\lam_k L+2\Delta\lam_k+\lam_k
\]
when comparing $x,y$ to two $\Delta\lam_k$-close points on the grid together with the first inequality of \eqref{Eq:GridEst}. Since the right hand side divided by $R_k-\Delta\lam_k$ converges to zero, the $\tilde\Phi_k$ form an RHA of $L$-Lipschitz maps.

A space $X$ is Lipschitz $n-1$-connected if there is a constant $l$ such that for every $m\in\{0,1,\ldots,n-1\}$, every $c$-Lipschitz map from the standard sphere $S^m$ into $X$ possesses a $lc$-Lipschitz extension to $B^{m+1}$ (the unit ball whose boundary is $S^m$). Given this, the restricted maps $\tilde\Phi_k$ may be extended directly through Theorem~1.5 in \cite{LanS}. Since the existence of a bicombing easily implies that a space is Lipschitz $(n-1)$-connected (even for arbitrary $n$), this yields a shorter proof of this lemma (and without requiring the bicombing to be reversible).
\end{proof}

In both Steps~2 and 3 we apply Fubini's theorem several times to real valued functions $h\in L^1(\Omega,\mathcal{L}_n)$ defined on (reasonably nice) subsets $\Omega\subset(\R^n,\|\cdot\|)$ of the normed space at hand. For a direction $v\in\di B$, let $C$ be the orthogonal projection of $\Omega$ to $v^\bot$, the Euclidean orthogonal complement of $v$. We then have
\[
\int_\Omega h\,d\mathcal{L}^n = \int_C \int_{(c+\R v)\cap\,\Omega} h(t)\,d\mathcal{L}^1_v(t)\,d\mathcal{L}^{n-1}_{v^\bot}(c) .
\]
$\mathcal{L}^n$ denotes the $n$-dimensional Lebesgue measure on $\R^n$. $\mathcal{L}^1_v$ shall be the push-forward of $(\R,\mathcal{L}^1)$ by an isometry $\R\to (c+\R v)\subset(\R^n,\|\cdot\|)$, thus it measures the length w.r.t.~the given norm $\|\cdot\|$ (which may differ from the Euclidean norm $\|\cdot\|_2$). Finally, in order for the product of the right hand side measures to yield $\mathcal{L}^n$, we have to rescale $\mathcal{L}^{n-1}$ on $v^\bot$ appropriately. The correct factor is $\|v\|_2/\|v\|$ and the resulting measure denoted by $\mathcal{L}^{n-1}_{v^\bot}$.

The next step employs a smoothing procedure by means of integration. Recall from Section~\ref{Chapter:Barycenters} the definition of Riemannian integrals for maps with values in spaces with reversible bicombings. The definition~\eqref{eq:folding} can be seen as a sort of convolution operation that evens out the defects of the given RHA.

\begin{Lem}[Step 2]\label{LemStep2} From an $L$-Lipschitz RHA for a space with reversible bicombing one may construct a $1$-Lipschitz RHA.
\end{Lem}

\begin{proof} Let $\Phi_k,R_k,\eps_k$ be as presumed and choose sequences $\mu_k,\lam_k$ such that $R_k\eps_k\prec\mu_k\prec\lam_k\prec R_k$, where $a_k\prec b_k$ means $a_k/b_k\to 0$. Observe that from $R_k\eps_k\prec\mu_k$ we obtain a sequence $\del_k\to 0$, such that $\|x-y\|\geq\mu_k$ implies $d(\Phi_k(x),\Phi_k(y))\leq(1+\del_k)\|x-y\|$. So on this intermediate scale, so to speak, we still have RHA-like estimates; we will use this later in the proof. We define a new RHA $\Psi_k$ on $B(R_k-\lam_k)$ by
\begin{equation}\label{eq:folding}
\Psi_k(x) := \int_{B(x,\lam_k)} \Phi_k \,d\mathcal{L}^n .
\end{equation}
As we defined the integral for probability measures only, the above integral is understood to be taken with respect to the measure appropriately normalized. Comparing $\Phi_k$ to $\Psi_k$ by means of \eqref{Eq:IntEst} yields $d(\Phi_k(x),\Psi_k(x))\leq L\lam_k$, thus $\Psi_k$ is in fact an RHA. In order to compare $\Psi_k(x)$ and $\Psi_k(y)$ we define a map $\phi_k$ from $B(x,\lam_k)$ to $B(y,\lam_k)$. For every $p\in B(x,\lam_k)$, the (possibly non-unit speed) line $t\mapsto p+t(y-x)$ intersects both $B(x,\lam_k)\!\setminus\!\operatorname{Interior}(B(y,\lam_k))$ and $B(y,\lam_k)\!\setminus\!\operatorname{Interior}(B(x,\lam_k))$ in a segment of equal length. Let $\phi_k$ map every such first segment onto the corresponding second by appropriate (and unique) translation in direction of $y-x$. On the remaining part of $B(x,\lam_k)$, which is the interior of $B(x,\lam_k)\cap B(y,\lam_k)$, $\phi_k$ shall be the identity. Elementary geometry tells us that $\phi_k$ is measure preserving (i.e.~$\mathcal{L}^n(\phi_k(A)) = \mathcal{L}^n(A)$) and $\Phi_k\circ\phi_k$ is integrable on $B(x,\lam_k)$. Moreover, by a direct comparison of the Riemann sums involved,
\[
\Psi_k(y)=\int_{B(x,\lam_k)} \Phi_k\circ\phi_k\,d\mathcal{L}^n
\]
and consequently
\begin{equation}\label{Eq:Psixy}
d(\Psi_k(x),\Psi_k(y)) \leq \frac{1}{\mathcal{L}^n(B(\lam_k))}\int_{B(x,\lam_k)} d(\Phi_k(t),\Phi_k(\phi_k(t)))\,d\mathcal{L}^n(t) .
\end{equation}
\begin{figure}[!ht]
\def\svgwidth{1.05\textwidth}
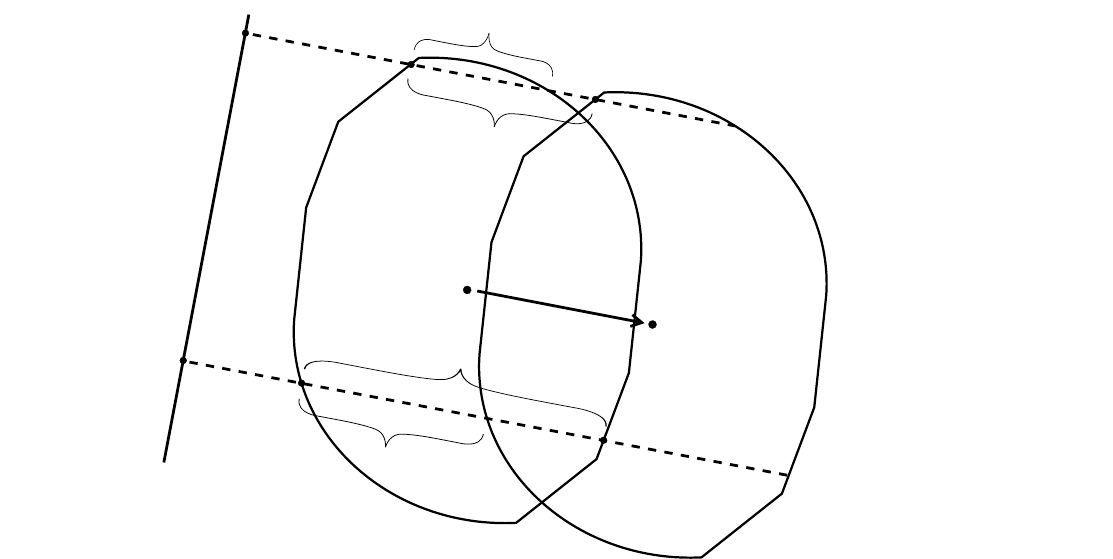
\caption{case analysis}
\end{figure}%
Now we fix a direction $v\in\di B$ and estimate \eqref{Eq:Psixy} for $y = x+sv$ and $s\leq \mu_k$. Let $C_{v,k}$ be the orthogonal projection of $B(x,\lam_k)$ to $v^\bot$.
We can rewrite the last integral as
\[
\frac{1}{\mathcal{L}^n(B(\lam_k))}\int_{C_{v,k}}\int_{(c+\R v)\cap B(x,\lam_k)}d(\Phi_k(t),\Phi_k(\phi_k(t)))\,d\mathcal{L}^1_v(t)\,d\mathcal{L}^{n-1}_{v^\bot}(c) .
\]
Let $l_{v,k}\colon C_{v,k}\to\R$ be the map that assigns the length of $(c+\R v)\cap B(x,\lam_k)$ to every $c$. We consider two cases. Case~1: $l_{v,k}(c)\geq \mu_k (\geq s)$. Then $d(\Phi_k,\Phi_k\circ\phi_k)$ is non-zero on an initial segment of $(c+\R v)\cap B(x,\lam_k)$ of length $s$ at most. There $\phi_k$ has displacement exactly $l_{v,k}(c)$ hence $d(\Phi_k,\Phi_k\circ\phi_k) \leq (1+\delta_k)l_{v,k}(c)$ where $\delta_k$ is the sequence mentioned in the beginning.
Case~2: $l_{v,k}(c)<\mu_k$. Again, the integrand is non-zero only on the initial segment of length at most $s$ and we estimate the displacement of $\phi_k$ --- which is $\max\{l_{v,k}(c),s\}$ --- by $\mu_k$. Consequently $d(\Phi_k,\Phi_k\circ\phi_k) \leq L\mu_k$ from the assumed Lipschitz continuity of $\Phi_k$. Now we split the integral according to these two cases and get for the first one
\[
\frac{1}{\mathcal{L}^n(B(\lam_k))}\int_{C_{v,k}}s(1+\delta_k)l_{v,k}(c)\,d\mathcal{L}^{n-1}_{v^\bot}(c) = (1+\delta_k)s = (1+\delta_k)\|x-y\|.
\]
The second part can be estimated generously by
\[
\frac{\mathcal{L}^{n-1}_{v^\bot}\!(C_{v,k})L\mu_k}{\mathcal{L}^n(B(\lam_k))}s = \frac{\lam_k^{n-1}\mathcal{L}^{n-1}_{v^\bot}\!(C_v)L\mu_k}{\lam_k^n\mathcal{L}^n(B)}s ,
\]
where $C_v$ is the orthogonal projection of the unit ball $B$. The measure of the cross sections $\mathcal{L}^{n-1}_{v^\bot}\!(C_v)$ is bounded by a constant independent of $v$, and therefore the above fraction (i.e.\ everything except the $s$) converges uniformly to zero for $k\to\infty$. Adding this to $\delta_k$ we arrive at $d(\Psi_k(x),\Psi_k(y)) \leq (1+\delta_k)\|x-y\|$ whenever $\|x-y\|\leq\mu_k$. Since the domains $B(R_k-\lam_k)$ are convex, the Lipschitz estimate extends to all pairs $x,y$. To conclude Step~2, we concatenate every $\Psi_k$ with the $1/(1+\delta_k)$-Lipschitz contraction $z\mapsto \sig_{\Psi_k(0)z}(1/(1+\delta_k))$ towards $\Psi_k(0)$ producing $1$-Lipschitz maps $\tilde\Psi_k$. And the estimate $d(\Psi_k(x),\tilde\Psi_k(x))\leq\delta_k/(1+\delta_k)d(\Psi_k(0),\Psi_k(x))$ together with, say, $d(\Psi_k(0),\Psi_k(x))\leq2R_k$ for all but finitely many $k$ ensures that $\tilde\Psi_k$ is still an RHA.
\end{proof}

We can now conclude the proof of Theorem~\ref{Thm:D} through the next and final Step~3. We use a kind of metric differentiation argument. Since we already made the maps of the RHA $1$-Lipschitz, we must now deal with the fact that they may collapse locally. But, intuitively, at points where the differential is large this should not be the case. Keep in mind the following example of an RHA, which may justify why we have to chose a density point $a$ in the proof below: $\Phi_k\colon B(k)\to \R^n,\, x\mapsto \max\{\|x\|-\sqrt{k},0\}\,x/\|x\|$, thus $\Phi_k$ maps $B(\sqrt{k})$ to $0$ and every other $x$ a $\sqrt{k}$-step towards $0$. Now $a=0$ is no appropriate choice to complete Step~3 as $\Phi_k$ converges uniformly on compact sets to the constant function $\Phi=0$.

\begin{Lem}[Step 3]\label{LemStep3} From a $1$-Lipschitz RHA for a proper space with cocompact isometry group, we may construct an isometric embedding of the normed space $(\R^n,\|\cdot\|)$.
\end{Lem}

\begin{proof} Let $\Phi_k\colon B(R_k)\to X$ be the presumed $RHA$ of $1$-Lipschitz maps 
and pick any direction $v\in\di B$. For every $x\in B(R_k)$ there is a minimal $t\in\R$ such that $p_k(x):=x+tv\in B(R_k)$ defining a map $p_k\colon B(R_k)\to B(R_k)$. Furthermore, set $H_k\colon B(R_k)\to\R, x\mapsto d(\Phi_k(p_k(x)),\Phi_k(x))$. This map is $1$-Lipschitz along straight segments $[p_k(x),q_k(x)]\subset B(R_k)$, where $q_k(x):=x+tv\in B(R_k)$ for a maximal $t$, and locally Lipschitz at interior points of $B(R_k)$. The latter claim holds since $p_k$ is Lipschitz there too, which may be seen by looking at the convex hull of an appropriate ball around an interior point $x$ and its projection $p_k(x)$. Therefore, and by virtue of Rademacher's theorem, we may differentiate $H_k$ in direction of $v$ yielding an $L^\infty(B(R_k),\mathcal{L}^n)$ derivative $\di_vH_k\colon B(R_k)\to [-1,1]$ defined almost everywhere with respect to the $n$-dimensional Lebesgue measure. From the fundamental theorem of calculus for Lipschitz functions we deduce
\[
\int_0^{d(p_k(x),q_k(x))} \di_v H_k\big(p_k(x)+tv\big) \,dt = H_k(q_k(x)) .
\]
For the maps $h_{v,k}(x)\colon B\to [-1,1], x\mapsto\di_v H_k(R_kx)$ --- that are just rescaled versions of the $\di_v H_k$ --- this translates to
\[
\int_{(x+\R v)\cap B} h_{v,k} \,d\mathcal{L}^1_v = \frac{d\big(\Phi_k(p_k(R_kx)),\Phi_k(q_k(R_kx))\big)}{R_k}.
\]
The right hand side is less or equal to the length of $(x+\R v)\cap B$ and at least that quantity minus $\eps_k$ (the $\eps_k$ being the one provided by the RHA). Let $C_v$ be the orthogonal projection of $B$ to $v^\bot$. Fubini's theorem gives
\[
\int_B h_{v,k}\,d\mathcal{L}^n = \int_{C_v} \int_{(x+\R v)\cap B} h_{v,k} \,d\mathcal{L}^1_v\,\mathcal{L}^{n-1}_{v^\bot}(x) ,
\]
thus we reach
\[
\mathcal{L}^n(B)-\eps_k\mathcal{L}^{n-1}_{v^\bot}(C_v)\leq \int_B h_{v,k}\,d\mathcal{L}^n \leq \mathcal{L}^{n}(B) .
\]
Since $h_{v,k}$ is bounded by $1$, we necessarily have $\mathcal{L}^n\!\left(\{x\in B \,|\, h_{v,k}(x)\geq 1-\delta \}\right)$ converging to $\mathcal{L}^n(B)$ for every $\delta>0$ when $k\to\infty$. As this holds for arbitrary directions $v$, we can pick a countable dense subset $D\subset \di B$ and, by a diagonal subsequence construction, assume that the set
\[
A:=\bigcap_{k\in\N,v\in D}\{x\in B\,|\, h_{v,k}(x) \geq c_{v,k} \}
\]
has positive $n$-dimensional Lebesgue measure for an appropriate choice of real numbers $c_{v,k}$ converging to $1$ for every fixed $v\in D$.
By virtue of Lebesgue's density theorem we next choose a density point $a\in A$ in the interior of $B$; by definition this is a point such that
\[
\lim_{\delta\to 0} \frac{\mathcal{L}^n(A\cap B(a,\delta))}{\mathcal{L}^n(B(a,\delta))} = 1 .
\]
Recall that $X$ was assumed to be proper. We make use of the cocompact isometry group as usual and can assume that $\Phi_k\circ (x\mapsto x+R_ka)$ converges uniformly on compacta to a $1$-Lipschitz map $\Phi\colon(\R^n,\|\cdot\|)\to X$. It remains to verify that $d(\Phi(x),\Phi(y))\geq\|x-y\|$ and it suffices to consider pairs where $x-y$ is a multiple of some $v\in D$. To obtain a contradiction, assume there is an $l>0$ with
\[
d(\Phi_k(x+R_ka),\Phi_k(y+R_ka))\leq(1-2l)\|x-y\|
\]
for all but finitely many $k\in\N$. Since the maps at hand are $1$-Lipschitz, a sufficiently small $r$ provides
\[
d(\Phi_k(x+s+R_ka),\Phi_k(y+s+R_ka))\leq(1-l)\|x+s-y-s\|
\]
for $s$ in $G:=\{ s \,|\, \|s\|\leq r,\, s\bot x-y \}$. Let $Z$ be the convex hull of $(x+G)\cup(y+G)$; a cylindric shape. From Fubini's theorem and the fundamental theorem of calculus we deduce
\begin{align*}
\int_{Z+R_ka} \di_v H_k \,d\mathcal{L}^n & = \int_{x+G+R_ka}\int_{(g+\R v)\cap(Z+R_ka)}
 \di_v H_k\,d\mathcal{L}^1_v \,d\mathcal{L}^{n-1}_{v^\bot}(g) \\
& \leq \int_{x+G+R_ka}(1-l)\|x-y\| \,d\mathcal{L}^{n-1}_{v^\bot}(g) = (1-l)\mathcal{L}^n(Z)
\end{align*}
for $k$ large enough. This translates to
\[
\int_{a+(1/R_k)Z} h_{v,k}\,d\mathcal{L}^n \leq (1-l)\mathcal{L}^n((1/R_k)Z)
\]
and, for $R>0$ (and $k$ large enough) such that $a+(1/R_k)Z\subset B(a,R/R_k)\subset B$, finally
\begin{align*}
\int_{B(a,R/R_k)} h_{v,k}\,d\mathcal{L}^n & \leq \mathcal{L}^n(B(a,R/R_k))-l\mathcal{L}^n((1/R_k)Z) \\
& = \mathcal{L}^n(B(a,R/R_k))\!\left[1-\frac{l\mathcal{L}^n(Z)}{\mathcal{L}^n(B(a,R))}\right] .
\end{align*}
But this contradicts the fact that $a$ is a density point of $A$. As for every $\alpha\in(0,1)$, we have 
\begin{align*}
\mathcal{L}^n\big(B(a,R/R_k)\cap \{x\in B\,|\, h_{v,k}(x) \geq c_{v,k} \}\big) & \geq
\mathcal{L}^n(B(a,R/R_k)\cap A) \\
& \geq\alpha\mathcal{L}^n(B(a,R/R_k))
\end{align*}
eventually, so $\alpha c_{v,k}\mathcal{L}^n(B(a,R/R_k))-(1-\alpha)\mathcal{L}^n(B(a,R/R_k))$ becomes a lower bound for the above integral at last. Therefore $d(\Phi(x),\Phi(y))\geq\|x-y\|$ which completes the proof.
\end{proof}

These three steps compose a proof of Theorem~\ref{Thm:D}, and
as a consequence we obtain the initially mentioned generalization 
of a part of Theorem~D in~\cite{Kle}. We refer to that paper for the 
definition of asymptotic cones.

\begin{Thm}\label{Thm:Dreally}
Let $X$ be a proper metric space with bicombing 
and cocompact isometry group.
Then for every positive integer $n$ the following are equivalent:
\begin{enumerate}
\item[\rm (1)]
There exists an isometric embedding of some $n$-dimensional normed space
into $X$.
\item[\rm (2)]
There is a quasi-isometric embedding of\/ $\R^n$ into $X$.
\item[\rm (3)]
There is a bi-Lipschitz embedding
of\/ $\R^n$ into some asymptotic cone $X_\omega$ of $X$.
\item[\rm (4)]
There is a sequence of subsets $Z_k$ of some asymptotic cone $X_\omega$ of $X$,
a sequence $0<r_k \to 0$, and a norm $\|\cdot\|$ on $\R^n$
such that $r_k^{-1}Z_k$ converges to the unit ball $B \subset (\R^n,\|\cdot\|)$
in the Gromov-Hausdorff topology.
\item[\rm (5)]
There exist a sequence of sets $S_k \subset X$, a sequence $0 < R_k \to \infty$, 
and a norm $\|\cdot\|$ on $\R^n$ such that $R_k^{-1}S_k$ converges to the 
unit ball $B \subset (\R^n,\|\cdot\|)$ in the Gromov-Hausdorff topology.
\end{enumerate}
\end{Thm}

\begin{proof}
The implications $(1) \Rightarrow (2) \Rightarrow (3)$ are clear, whereas
$(5) \Rightarrow (1)$ is just Theorem~\ref{Thm:D}. The implication 
$(4) \Rightarrow (5)$ is shown as in the last paragraph on p.~455 
in~\cite{Kle}. The implication $(3) \Rightarrow (4)$ is proved by a metric 
differentiation argument, compare Proposition~10.18 in~\cite{Kle} and 
Corollary~2.2 in~\cite{Wen}.
\end{proof}


\bigskip\noindent
D.~Descombes ({\tt dominic.descombes@math.ethz.ch}),\\
Department of Mathematics, ETH Zurich, 8092 Zurich, Switzerland 


\end{document}